\documentclass[11pt,twoside]{article}

\usepackage{amssymb,amsmath,amsfonts,amsthm,mathrsfs}
\usepackage{enumerate}
\usepackage[colorlinks=true, pdfstartview=FitV, linkcolor=blue, citecolor=blue, urlcolor=blue]{hyperref}
\usepackage{graphicx,color}
\usepackage{cite}
\usepackage{indentfirst}

\allowdisplaybreaks

\def\rn{{\mathbb R^n}}  \def\sn{{\mathbb S^{n-1}}}
\def\co{{\mathcal C_\Omega}}

\newtheorem{thm}{Theorem}
\newtheorem{lem}{Lemma}
\newtheorem*{thmA}{Theorem A}	\newtheorem*{thmB}{Theorem B}	
\newtheorem{rem}{Remark}
\newtheorem{que}{Question}
\newtheorem{cor}{Corollary}

\numberwithin{thm}{section}
\numberwithin{lem}{section}
\numberwithin{equation}{section}
\numberwithin{rem}{section}
\numberwithin{cor}{section}

\begin{document}

\title{\bf\Large On the Bounds of Weak $(1,1)$ Norm of Hardy-Littlewood Maximal Operator with $L\log L(\sn)$ Kernels
\footnotetext{\hspace{-0.35cm} 2020 {\it Mathematics Subject Classification}. Primary 42B20; Secondary 42B25. \endgraf
{\it Key words and phrases}. The maximal operator, $L\log L(\sn)$ rough kernel, upper bound, lower bound, limiting weak-type behaviors.
\endgraf
This project is partially supported by NSFC
(Nos. 11871101,1217011127), 111 Project and the National Key Research and Development Program of China (Grant No. 2020YFA0712900).}}

\date{}

\author{ Moyan Qin, Huoxiong Wu, Qingying Xue \footnote{Corresponding author, E-mail: \texttt{qyxue@bnu.edu.cn}}}

\maketitle

\vspace{-0.8cm}

\begin{center}
\begin{minipage}{13cm}
{\small {\bf Abstract}\quad
Let $\Omega\in L^1{(\sn)}$, be a function of homogeneous of degree zero{, and }$M_\Omega$ be the Hardy-Littlewood maximal operator associated with $\Omega$ defined by $M_\Omega(f)(x) = \sup_{r>0}\frac1{r^n}\int_{|x-y|<r}|\Omega(x-y)f(y)|dy.$ It was shown by  Christ and Rubio de Francia that $\|M_\Omega(f)\|_{L^{1,\infty}(\rn)} \le C(\|\Omega\|_{L\log L(\sn)}+1)\|f\|_{L^1(\rn)}$ provided $\Omega\in L\log L {(\sn)}$.
In this paper, we show that, if $\Omega\in L\log L(\sn)$, then for all $f\in L^1(\rn)$, $M_\Omega$ enjoys the limiting weak-type behaviors that \[\lim_{\lambda\to0^+}\lambda|\{x\in\rn:M_\Omega(f)(x)>\lambda\}| = n^{-1}\|\Omega\|_{L^1(\sn)}\|f\|_{L^1(\rn)}.\]
{This removes the smoothness restrictions on the kernel $\Omega$, such as Dini-type conditions, in previous results. To prove our result, we present a new upper bound of $\|M_\Omega\|_{L^1\to L^{1,\infty}}$, which essentially improves the upper bound $C(\|\Omega\|_{L\log L(\sn)}+1)$ given by  Christ and Rubio de Francia. As a consequence, the upper and lower bounds of $\|M_\Omega\|_{L^1\to L^{1,\infty}}$ are obtained for $\Omega\in L\log L {(\sn)}$.}
}
\end{minipage}
\end{center}

\vspace{0.2cm}


\section{Introduction}\label{s1}


As one of the fundamental operators in modern analysis, the Hardy-Littlewood maximal operator
 \[M(f)(x) = \sup_{r>0}\frac1{r^n}\int_{|x-y|<r}|f(y)|dy,\]
has played very important roles in several fields, such as Harmonic analysis, ergodic theory and index theory.
It was well known that the almost everywhere convergence of some important operators is usually closely related to whether the associated maximal operators satisfy weak type inequalities or not. The Hardy-Littlewood maximal operator and its purpose in differentiation were first introduced by Hardy and Littlewood \cite{HL1930} on $\mathbb R$, and later extended and developed by Wiener \cite{W1939} on $\rn$. It was shown that $M$ is of weak $(1,1)$ type and $L^p$ bounded for $p>1$.

In 1956,  Calder\'on and Zygmund \cite{CZ1956} considered the following rough Hardy-Littlewood maximal operator
 \[M_\Omega(f)(x) = \sup_{r>0}\frac1{r^n}\int_{|x-y|<r}|\Omega(x-y)f(y)|dy,\]
where $\Omega$ is a function of homogeneous of degree zero and $\Omega\in L^1(\rn)$. It was shown in  \cite{CZ1956}  that $M_\Omega$ is bounded on $L^p(\rn)$ for all $p>1$ by using the method of rotation. The weak $(1,1)$ boundedness of $M_\Omega$ was given by  Christ \cite{C1988} when $\Omega\in L^q(\sn)$ for $q>1$. Later on, Christ and  Rubio de Francia \cite{CR1988} extended the kernel condition $\Omega\in L^q(\sn)$ to a more larger space $L\log L(\sn)$ (since $L^q(\sn)\subsetneq L\log L(\sn) \subsetneq L^1(\sn)$).
Their result can be stated as follows:
\begin{thmA}[\cite{CR1988}]\label{thmA}
	For all $n\ge2$ and $\Omega\in L\log L(\sn)$, $M_\Omega$ is of weak type $(1,1)$, and enjoys
	\begin{equation}\label{eq1.1}
		\|M_\Omega(f)\|_{L^{1,\infty}(\rn)} \le C(\|\Omega\|_{L\log L(\sn)}+1)\|f\|_{L^1(\rn)}.
	\end{equation}
\end{thmA}

\begin{rem} It is worthy pointing that, similar result like (\ref{eq1.1}) also holds for the singular integral $T_\Omega$ with rough homogeneous kernels \cite{CZ1956}, and the constant on the right side is also $C(\|\Omega\|_{L\log L(\sn)}+1)$.
\end{rem}
Now let's turn to the best constants problem for the weak norm inequalities of some important operators in Harmonic analysis. This belongs to less fine problems and has attracted lots of attentions. For example, in one deminsion $\mathbb R$, {Davis \cite{D1974} obtained the best constant of weak-type $(1,1)$ for Hilbert transform; Melas \cite{M2003} proved that $\|M\|_{L^1\to L^{1,\infty}}=\frac{11+\sqrt{61}}{11}$ for the Hardy-Littlewood maximal operator $M$; Grafakos and Montgomery-Smith \cite{GM1997} showed that the $(p,p)$ norm of the uncentered Hardy-Littlewood maximal operator is the unique positive solution of the equation $(p-1)x^p-px^{p-1}=1$; Grafakos and Kinnunen \cite{GK1998} pointed out that the upper bound of the weak $(1,1)$ norm  is $2$ for uncentered Hardy-Littlewood maximal operator, and the results in \cite{GK1998} can be extended to more general measure space. However, when the dimension $n$ is bigger than $1$, things become more subtle. Even for the well-known Riesz transform, there is no such information.

To explore the lower bounds of $\|M\|_{L^1\to L^{1,\infty}}$, in 2006, Janakiraman \cite{J2006} considered the following limiting weak-type behavior for the Hardy-Littlewood maximal operator $M$:
\begin{equation}\lim_{\lambda\to0^+}\lambda|\{x\in\rn:M(f)(x)>\lambda\}| = \|f\|_{L^1(\rn)},\end{equation}
which yields the lower bound in the way that
\begin{equation}\label{eq1.2}
	\|M\|_{L^1(\rn)\to L^{1,\infty}(\rn)} \ge 1.
\end{equation}

Although (\ref{eq1.2}) can be concluded from a geometric point of view by simply taking $f=\chi_{B(0,1)}$, the limiting weak-type behavior still offers us a new viewpoint in obtaining the lower bounds of the best constants.

In 2017, Ding and Lai \cite{DL2017-2} extended Janakiraman's results to $M_\Omega$ with the kernel $\Omega$ satisfies the $L^1$-Dini condition{, see also \cite{GHW} for more general limiting weak-type behaviors}. Recall that,  $\Omega\in L^1(\sn)$ is said to satisfy the the $L^1$-Dini condition if
{\[\int_0^1\frac{\omega(\delta)}\delta d\delta<\infty,\]
where $\omega(\delta):=\sup_{\|\rho\|\le\delta}\int_\sn|\Omega(\rho\theta)-\Omega(\theta)|d\sigma(\theta)$ for $\delta>0$, here $\rho$ denotes a rotation on $\rn$ and $\|\rho\|:=\sup\{|\rho \theta-\theta|:\,\theta\in\sn\}$.}

\begin{thmB}[\cite{DL2017-2}]\label{thmb1}
	Let $\Omega$ satisfy the $L^1$-Dini condition and $f\in L^1(\rn)$. Then
	\[\lim_{\lambda\to0^+}\lambda|\{x\in\rn:M_\Omega(f)(x)>\lambda\}| = n^{-1}\|\Omega\|_{L^1(\sn)}\|f\|_{L^1(\rn)}.\]
\end{thmB}
As we mentioned before, {in Theorem A,} $\Omega\in L\log L(\sn)$ is a sufficient condition to guarantee the weak-type $(1,1)$ boundedness of $M_\Omega$. 
{Note that
\[L^1-\text{Dini}\subsetneq L\log L(\sn) \subsetneq L^1(\sn),\]
and the result of Ding and Lai \cite{DL2017-2} {was} obtained under $L^1$-Dini condition.
It is natural to ask the following question:
  \begin{que}
  Does $M_\Omega$ still enjoy the limiting weak-type behavior when $\Omega\in L\log L(\sn)$?
  \end{que}

One of the main purpose in this paper is to address the question above. Our first result can be formulated as follows.}




\begin{thm}\label{thm2}
Let $\Omega\in L\log L(\sn)$. Then for all $f\in L^1(\rn)$, it holds that
\begin{enumerate}
	\item[\rm{(i)}] $\displaystyle\lim_{\lambda\to0^+}\lambda|\{x\in\rn:M_\Omega(f)(x)>\lambda\}| = n^{-1}\|\Omega\|_{L^1(\sn)}\|f\|_{L^1(\rn)};$
\item[\rm{(ii)}] $\displaystyle\lim_{\lambda\to0^+}\lambda\big|\big\{x\in\rn:\big|M_\Omega(f)(x) - \|f\|_{L^1(\rn)}|\Omega(x)||x|^{-n}\big|>\lambda\big\}\big|=0.$
\end{enumerate}
\end{thm}

\begin{rem}
	Note that, in Theorem \ref{thm2}, we removed the smoothness condition and only assumed the kernel $\Omega \in L\log L(\sn)$. Therefore, Theorem \ref{thm2} essentially improved the result in Theorem {B}.
\end{rem}

{Furthermore, in order to show Theorem \ref{thm2}, it requires that the upper bound of $\|M_\Omega\|_{L^1\to L^{1,\infty}}$ tends to $0$ as $\|\Omega\|_{L\log L(\sn)} \to 0$. However, it is clear that the upper bound in (\ref{eq1.1}) of Theorem A does not satisfies this requirement. This naturally  leads to the following question.

\begin{que}
Is it possible to give an improvement of the upper bound $\|\Omega\|_{L\log L(\sn)}+1$ in the inequality (\ref{eq1.1}) ?
\end{que}

This question will be affirmed by our next theorem.}


\begin{thm}\label{thm1}
	Suppose $\Omega\in L\log L(\sn)$. Then $M_\Omega$ is of weak type $(1,1)$, and enjoys
	\[\|M_\Omega\|_{L^{1,\infty}(\rn)} \le C\co\|f\|_{L^1(\rn)},\]
		where
	\[\co = \|\Omega\|_{L\log L(\sn)} + \int_\sn|\Omega(\theta)|\Big(1+\log^+\frac{|\Omega(\theta)|}{\|\Omega\|_{L^1(\sn)}}\Big)d\sigma(\theta).\]
\end{thm}

\begin{rem}
Note that $\co \le 3(\|\Omega\|_{L\log L(\sn)}+1)$. Hence, Theorem \ref{thm1} is an essential improvement of Theorem A. Moreover, $\co$ has a nice property that it tends to $0$ as $\|\Omega\|_{L\log L(\sn)}$ tends to $0$. This improvement plays an important role in proving the limiting weak-type behaviors of $M_\Omega$ in Theorem \ref{thm2}{, and it has its own interest in bounding the range of $\|M_\Omega\|_{L^1\to L^{1,\infty}}$.}
\end{rem}

As a consequence of Theorem \ref{thm2} and Theorem \ref{thm1}, we {obtain} the upper and lower bound of the weak norm of $M_\Omega$.

\begin{cor}
Let $\Omega\in L\log L(\sn)$. Then it holds that
\[n^{-1}\|\Omega\|_{L^1(\sn)} \le \|M_\Omega\|_{L^1(\rn)\to L^{1,\infty}(\rn)} \le C\co.\]
\end{cor}
\begin{rem}
{Very recently, the corresponding results for singular integrals with rough kernels $\Omega\in L\log L(\sn)$ were established in \cite{QWX2021}.}
	
\end{rem}
The organization of this paper is as follows.  Section \ref{s2}	will be devoted to proving Theorem \ref{thm1}. In Section \ref{s3}, the proof of Theorem \ref{thm2} will be demonstrated.

Throughout this paper, the letter C will stand for positive constants not necessarily the same one
at each occurrence but independent of the essential variables
\vspace{0.1cm}

\vspace{0.1cm}


\section{Proof of Theorem \ref{thm1}.}\label{s2}

We need to {recall} some notations first. Let $\sn$ be the unit sphere on the Euclidean space $\rn$, $d\sigma(\cdot)$ be the induced Lebesgue measure on $\sn$, and $L\log L(\sn)$ be the function space on $\sn$ contains all $\Omega$ satisfying
\[\|\Omega\|_{L\log L(\sn)} = \int_\sn|\Omega(\theta)|\log(\text{e}+|\Omega(\theta)|)d\sigma(\theta) < \infty.\]
Set $\phi$ be a smooth, radial, nonnegative function in $\rn$ which supported in $\{x:1/2\le|x|\le4\}$, satisfies $\phi\le1$ and
\[\phi(x)\equiv1\quad\text{for }1\le|x|\le2.\]


In order to prove Theorem \ref{thm1}, we need the following lemma which is attributed to Christ and Rubio de Francia \cite{CR1988}.

\begin{lem}[\cite{CR1988}]\label{lem21}
	Define
	\[L^x(\Omega)(\theta) = \int_{\mathbb R}\phi(x-t\theta)\phi(t\theta)\Omega(x-t\theta)dt,\]
	and
	\[T(\Omega)(\theta) = \int_\rn L^x(\Omega)(\theta)\widetilde B_{-s}(x)dx,\]
	where $\widetilde B_{-s} = \sum\limits_{\widetilde Q}b_{\widetilde Q}$ satisfies
	\begin{enumerate}
		\item[{\rm (i)}] There exists $x_0\in\rn$, for all $\widetilde Q$, $x_0+\widetilde Q$ is a dyadic cube with length $2^{-s}$;
		\item[{\rm (ii)}] $\text{supp }b_{\widetilde Q} \subset \widetilde Q$, $\int b_{\widetilde Q} = 0$ and $\|b_{\widetilde Q}\|_{L^1(\rn)} \le 2^{n+1}\lambda|\widetilde Q|$.
	\end{enumerate}
		Then there exists $\gamma>0$, for all $\Omega\in L^\infty(\sn)$, all $s>-3$ and all $\widetilde B_{-s}$,
	\[\|T(\Omega)\|_{L^1(\sn)} \le C2^{-\gamma s}\lambda\|\Omega\|_{L^\infty(\sn)}.\]
\end{lem}

\begin{proof}[Proof of Theorem \ref{thm1}] The main idea of this proof is from Christ and Rubio de Francia \cite{CR1988}, Ding and Lai \cite{DL2019}. Without losing generality, we assume $\Omega$ and $f$ are both nonnegative and $\co>0$. For readability, we split the proof into four steps.

\vspace{0.2cm}


\noindent{\bf Step 1: Control of $M_\Omega(f)$.}

Since $\phi$ is smooth, then there exists $\alpha\in(0,1)$ such that
\[\phi(x)>\frac12\quad\text{for }2^{-\alpha}\le|x|\le2.\]
Define
\[\phi_j(x) = \frac1{2^{jn}}\phi(2^{-j}x).\]
thus
\[\phi_j(x)>\frac1{2^{jn+1}}\quad\text{for }2^{j-\alpha}\le|x|\le2^{j+1}.\]

Let $\Omega_j=\phi_j\Omega$. For any fixed $x\in\rn$, by the definition of $M_\Omega(f)(x)$, there exists $r_x\in(0,\infty)$ such that
\[M_\Omega(f)(x) \le \frac{2^{\alpha n/2}}{r_x^n}\int_{B(x,r_x)}\Omega(x-y)f(y)dy.\]
Let $j_x\in\mathbb Z$, satisfy $2^{j_x}\le r_x<2^{j_x+1}$, then
\begin{align*}
	M_\Omega(f)(x) \le& \frac{2^{\alpha n/2}}{2^{j_xn}}\int_{B(x,2^{j_x+1})}\Omega(x-y)f(y)dy \\
	=& \frac{2^{1+\alpha n/2}}{2^{j_xn+1}}\int_{B(x,2^{j_x+1})\backslash B(x,2^{j_x-\alpha})}\Omega(x-y)f(y)dy \\
	&+ \frac1{2^{\alpha n/2}2^{(j_x-\alpha)n}}\int_{B(x,2^{j_x-\alpha})}\Omega(x-y)f(y)dy \\
	\le& 2^{1+\alpha n/2}\int_\rn\phi_{j_x}(x-y)\Omega(x-y)f(y)dy + \frac1{2^{\alpha n/2}}M_\Omega(f)(x).
\end{align*}
Therefore, for all $x\in\rn$,
\begin{equation}\label{eq3.1}
	M_\Omega(f)(x) \le C\Omega_{j_x}\ast f(x) \le C\sup_j\Omega_j\ast f(x).
\end{equation}


\noindent{\bf Step 2: C-Z decomposition.}

For $\lambda>0$, applying C-Z decomposition at level $\frac{n\lambda}{2^{3n+1}\co}$, we have the following conclusions:
\begin{enumerate}
	\item[(cz-f)] $f = g+b = g+\sum\limits_{Q\in\mathcal Q}b_Q$;
	\item[(cz-g)] $\|g\|_{L^\infty(\rn)} \le \frac{n\lambda}{2^{2n+1}\co}$;
	\item[(cz-Q)] $\mathcal Q$ is a countable set of disjoint dyadic cubes. Let $E=\bigcup_{Q\in\mathcal Q}Q$, then $|E| \le \frac{2^{3n+1}\co}{n\lambda}\|f\|_{L^1(\rn)}$;
	\item[(cz-b)] For each $Q\in\mathcal Q$, we have $\text{supp }b_Q\subset Q$, $\int b_Q=0$ and $\|b_Q\|_{L^1(\rn)} \le \frac{n\lambda}{2^{2n}\co}|Q|$. Thus by (cz-Q), we have $\|b\|_{L^1(\rn)} \le 2^{2n+1}\|f\|_{L^1(\rn)}$.
\end{enumerate}
Therefore it follows from (\ref{eq3.1}) that
\begin{equation}\label{eq3.2}
	\begin{aligned}
		& |\{x:M_\Omega(f)(x)>\lambda\}| \le C|\{x:\sup_j\Omega_j\ast f(x)>\lambda\}| \\
		&\le C\Big(\big|\big\{x:\sup_j|\Omega_j\ast g(x)|>\frac\lambda2\big\}\big| + \big|\big\{x:\sup_j|\Omega_j\ast b(x)|>\frac\lambda2\big\}\big|\Big).
	\end{aligned}
\end{equation}

As a matter of fact, for all $j\in\mathbb Z$ and $x\in\rn$, by (cz-g), we have
\begin{align*}
	|\Omega_j\ast g(x)| \le& \int_\rn\phi_j(y)\Omega(y)|g(x-y)|dy \le \|g\|_{L^\infty(\rn)}\int_{\{1/2\le|y|\le4\}}\Omega(y)dy \\
	=& \|g\|_{L^\infty(\rn)}\|\Omega\|_{L^1(\sn)}\frac{4^n-1/2^n}n \le \frac{n\lambda}{2^{2n+1}\co}\cdot\frac{4^n\co}n \le \frac\lambda2,
\end{align*}
which indicates that
\[\big\{x:\sup_j|\Omega_j\ast g(x)|>\frac\lambda2\big\}=\emptyset.\]
Combining with (\ref{eq3.2}), we can deduce that
\[|\{x:M_\Omega(f)(x)>\lambda\}| \le C\big|\big\{x:\sup_j|\Omega_j\ast b(x)|>\frac\lambda2\big\}\big|.\]

Denote $E^\ast = \bigcup_{Q\in\mathcal Q}2Q$. Then (cz-Q) implies that
\[|E^\ast| \le \sum_{Q\in\mathcal Q}|2Q| = 2^n\sum_{Q\in\mathcal Q}|Q| = 2^n|E| \le \frac{2^{4n+1}\co}{n\lambda}\|f\|_{L^1(\rn)}.\]
Therefore
\[|\{x:M_\Omega(f)(x)>\lambda\}| \le C\frac\co\lambda\|f\|_{L^1(\rn)} + C\big|\big\{x\in(E^\ast)^c:\sup_j|\Omega_j\ast b|(x)>\frac\lambda2\big\}\big|.\]

Let $B_s = \sum\limits_{Q\in\mathcal Q,|Q|=2^{sn}}b_Q$. Then we have
\[\sup_j|\Omega_j\ast b(x)| = \sup_j|\sum_s\Omega_j\ast B_{j-s}(x)| \le \sup_j\sum_s|\Omega_j\ast B_{j-s}(x)| \le \sum_s\sup_j|\Omega_j\ast B_{j-s}(x)|.\]
Notice that for all $s\le-3$ and $x\in(E^\ast)^c$, $\Omega_j\ast B_{j-s} = 0$. Therefore
\begin{equation}\label{eq3.3}
	\begin{aligned}
		|\{x:M_\Omega(f)(x)>\lambda\}| \le& C\frac\co\lambda\|f\|_{L^1(\rn)} \\
		&+ C\big|\big\{x\in(E^\ast)^c:\sum_{s>-3}\sup_j|\Omega_j\ast B_{j-s}|(x)>\frac\lambda2\big\}\big|.
	\end{aligned}
\end{equation}


\noindent{\bf Step 3: Split of $\Omega_j$.}

Let
\begin{align*}
	\mathcal D_s &= \{\theta\in\sn:\Omega(\theta)\ge2^{\gamma s/3}\|\Omega\|_{L^1(\sn)}\}; \\
	E_s &= \big\{x\in\rn:\frac x{|x|}\in\mathcal D_s\big\}.
\end{align*}
and $\Omega_j^s = \phi_j\Omega\chi_{E_s}, K_j^s = \Omega_j-\Omega_j^s$. Then it is easy to see that
\begin{equation}\label{eq3.4}
	\begin{aligned}
		& \big|\big\{x\in(E^\ast)^c:\sum_{s>-3}\sup_j|\Omega_j\ast B_{j-s}|(x)>\frac\lambda2\big\}\big| \\
		&\le \big|\big\{x\in(E^\ast)^c:\sum_{s>-3}\sup_j|\Omega_j^s\ast B_{j-s}|(x)>\frac\lambda4\big\}\big| \\
		&\quad + \big|\big\{x\in(E^\ast)^c:\sum_{s>-3}\sup_j|K_j^s\ast B_{j-s}|(x)>\frac\lambda4\big\}\big|.
	\end{aligned}
\end{equation}

For the first term on the right side of (\ref{eq3.4}), from a basic fact that
\begin{align*}
	\|\Omega_j^s\|_{L^1(\rn)} &= \int_\rn\phi_j(y)\Omega(y)\chi_{E_s}(y)dy = \int_\rn\phi(y)\Omega(y)\chi_{E_s}(y)dy \\
	&\le \int_{\{1/2\le|y|\le4\}}\Omega(y)\chi_{E_s}(y)dy = \int_{\mathcal D_s}\int_{1/2}^4\Omega(\theta)t^{n-1}dtd\sigma(\theta) \\
	&= \frac{4^n-1/2^n}n\int_{\mathcal D_s}\Omega(\theta)d\sigma(\theta),
\end{align*}
thus, by (cz-b) and Chebyshev's inequality, we get
\begin{equation}\label{eq3.5}
	\begin{aligned}
		& \big|\big\{x\in(E^\ast)^c:\sum_{s>-3}\sup_j|\Omega_j^s\ast B_{j-s}|(x)>\frac\lambda4\big\}\big| \\
		&\le \frac 4\lambda\big\|\sum_{s>-3}\sup_j|\Omega_j^s\ast B_{j-s}|\big\|_{L^1(\rn)} \le \frac 4\lambda\sum_{s>-3}\sum_j\big\|\Omega_j^s\ast B_{j-s}\big\|_{L^1(\rn)} \\
		&\le \frac 4\lambda\sum_{s>-3}\sum_j\|\Omega_j^s\|_{L^1(\rn)}\|B_{j-s}\|_{L^1(\rn)} \\
		&\le \frac C\lambda\sum_{s>-3}\sum_j\|B_{j-s}\|_{L^1(\rn)}\int_{\mathcal D_s}\Omega(\theta)d\sigma(\theta) \\
		&= \frac C\lambda\|b\|_{L^1(\rn)}\sum_{s>-3}\int_{\{2^{\gamma s/3}\le\frac{\Omega(\theta)}{\|\Omega\|_{L^1(\rn)}}\le2^{\gamma(s+1)/3}\}}(s+3)\Omega(\theta)d\sigma(\theta) \\
		&\le \frac C\lambda\|f\|_{L^1(\rn)}\int_\sn\Omega(\theta)\Big(\log^+\frac{\Omega(\theta)}{\|\Omega\|_{L^1(\rn)}}+1\Big)d\sigma(\theta) \\
		&\le C\frac\co\lambda\|f\|_{L^1(\rn)}.
	\end{aligned}
\end{equation}

For the second term on the right side of (\ref{eq3.4}), we claim that for any fixed $s>-3$, it holds that
\begin{equation}\label{eq3.6}
	\big\|\sup_j|K_j^s\ast B_{j-s}|\big\|_{L^2(\rn)}^2 \le C2^{-\gamma s/3}\lambda\|\Omega\|_{L^1(\sn)}\|b\|_{L^1(\rn)},
\end{equation}
which will be proved in Step 4. Now applying (cz-b) and Chebyshev's inequality again, we can deduce that,
\begin{equation}\label{eq3.7}
	\begin{aligned}
		& \big|\big\{x\in(E^\ast)^c:\sum_{s>-3}\sup_j|K_j^s\ast B_{j-s}(x)|>\frac\lambda4\big\}\big| \\
		&\le \frac{16}{\lambda^2}\Big\|\sum_{s>-3}\sup_j|K_{j,v}^s\ast B_{j-s}|\Big\|_{L^2(\rn)}^2 \\
		&\le \frac{64}{\lambda^2}\Big(\sum_{s>-3}\big\|\sup_j|K_{j,v}^s\ast B_{j-s}|\big\|_{L^2(\rn)}\Big)^2 \\
		&\le \frac C\lambda\|\Omega\|_{L^1(\sn)}\|f\|_{L^1(\rn)} \le C\frac\co\lambda\|f\|_{L^1(\rn)}.
	\end{aligned}
\end{equation}

Combining (\ref{eq3.4}), (\ref{eq3.5}), (\ref{eq3.7}) with (\ref{eq3.3}), we conclude that
\[|\{x:M_\Omega(f)(x)>\lambda\}| \le C\frac\co\lambda\|f\|_{L^1(\rn)}.\]
{This is the desired conclusion. It remains to prove (\ref{eq3.6}).}

\vspace{0.2cm}


\noindent{\bf Step 4: Proof of (\ref{eq3.6}).}

Once we show that
\begin{equation}\label{eq3.8}
	\|K_j^s\ast B_{j-s}\|_{L^2(\rn)}^2 \le C2^{-\gamma s/3}\lambda\|\Omega\|_{L^1(\sn)}\|B_{j-s}\|_{L^1(\rn)},
\end{equation}
it follows immediately that
\begin{align*}
	\big\|\sup_j|K_j^s\ast B_{j-s}|\big\|_{L^2(\rn)}^2 &\le \Big\|\Big(\sum_j|K_j^s\ast B_{j-s}|^2\Big)^{1/2}\Big\|_{L^2(\rn)}^2 \\
	&= \sum_j\|K_j^s\ast B_{j-s}\|_{L^2(\rn)}^2 \\
	&\le C2^{-\gamma s/3}\lambda\|\Omega\|_{L^1(\sn)}\sum_j\|B_{j-s}\|_{L^1(\rn)} \\
	&= C2^{-\gamma s/3}\lambda\|\Omega\|_{L^1(\sn)}\|b\|_{L^1(\rn)},
\end{align*}
where the last equality is from (cz-b).

To see (\ref{eq3.8}), we rewrite $\|K_j^s\ast B_{j-s}\|_{L^2(\rn)}^2$ as follows
\begin{align*}
	\|K_j^s\ast B_{j-s}\|_{L^2(\rn)}^2 &= \int_\rn K_j^s\ast B_{j-s}(x)\cdot K_j^s\ast B_{j-s}(x)dx \\
	&= \int_\rn\mathcal F^{-1}(K_j^s)(x)\cdot\mathcal F^{-1}(B_{j-s})(x)\cdot\mathcal F\big(K_j^s\ast B_{j-s}\big)(x)dx \\
	&= \int_\rn\mathcal F(\widetilde K_j^s)(x)\cdot\mathcal F^{-1}(B_{j-s})(x)\cdot\mathcal F\big(K_j^s\ast B_{j-s}\big)(x)dx \\
	&= \int_\rn\mathcal F\big(\widetilde K_j^s\ast K_j^s\ast B_{j-s}\big)(x)\cdot\mathcal F^{-1}(B_{j-s})(x)dx \\
	&= \int_\rn\widetilde K_j^s\ast K_j^s\ast B_{j-s}(x)\cdot B_{j-s}(x)dx,
\end{align*}
where $\mathcal F$ and $\mathcal F^{-1}$ are the Fourier transform and the inverse Fourier transform, and $\widetilde K_j^s(x) = K_j^s(-x)$. Then we only need to show that
\[\|\widetilde K_j^s\ast K_j^s\ast B_{j-s}\|_{L^\infty(\rn)} \le C2^{-\gamma s/3}\lambda\|\Omega\|_{L^1(\sn)}.\]

Furthermore, the problem can be reduced to prove that
\begin{equation}\label{eq3.9}
	|\widetilde K_0^s\ast K_0^s\ast\widetilde B_{-s}(0)| \le C2^{-\gamma s/3}\lambda\co\|\Omega\|_{L^1(\sn)},
\end{equation}
where $\widetilde B_{-s}$ satisfies the conditions (i) and (ii) in Lemma \ref{lem21}. Suppose (\ref{eq3.9}) holds, then for all $x\in\rn$, denote
\[B_{j-s}^j(z) = B_{j-s}(2^jz),\qquad B_{j-s}^{j,x}(z) = B_{j-s}^j(z+2^{-j}x).\]
It is not hard to verify that $\frac{2^{2n}\co}nB_{j-s}^{j,x}$ is one of $\widetilde B_{-s}$. Therefore
\begin{align*}
	|\widetilde K_j^s\ast K_j^s\ast B_{j-s}(x)| &= \Big|\int_\rn\int_\rn\frac1{2^{2jn}}\widetilde K_0^s(2^{-j}(x-y-z))K_0^s(2^{-j}y)B_{j-s}(z)dydz\Big| \\
	&= \Big|\int_\rn\int_\rn\widetilde K_0^s(2^{-j}x-y-z)K_0^s(y)B_{j-s}^j(z)dydz\Big| \\
	&= \Big|\int_\rn\int_\rn\widetilde K_0^s(0-y-z)K_0^s(y)B_{j-s}^{j,x}(z)dydz\Big| \\
	&= \frac n{2^{2n}\co}\Big|\widetilde K_0^s\ast K_0^s\ast\frac{2^{2n}\co}nB_{j-s}^{j,x}(0)\Big| \le C2^{-\varepsilon s/3}\lambda\|\Omega\|_{L^1(\sn)}.
\end{align*}

Back to (\ref{eq3.9}), we split $\Omega\chi_{\rn\backslash E_s}$ by an even function $\Omega_s'$ and an odd function $\Omega_s''$ as follows:
\begin{align*}
	\Omega(x)\chi_{\rn\backslash E_s}(x) =& \frac{\Omega(x)\chi_{\rn\backslash E_s}(x) + \Omega(-x)\chi_{\rn\backslash E_s}(-x)}2 \\
	&+ \frac{\Omega(x)\chi_{\rn\backslash E_s}(x) - \Omega(-x)\chi_{\rn\backslash E_s}(-x)}2 \\
	=&: \Omega_s'(x) + \Omega_s''(x).
\end{align*}
Obviously $\Omega_s'$ and $\Omega_s''$ are homogeneous of degree $0$ and enjoy
\begin{equation}\label{eq3.10}
	\begin{aligned}
		\|\Omega_s'\|_{L^\infty(\sn)} &\le \|\Omega\chi_{\sn\backslash\mathcal D_s}\|_{L^\infty(\sn)} \le 2^{\gamma s/3}\|\Omega\|_{L^1(\sn)}; \\
		\|\Omega_s''\|_{L^\infty(\sn)} &\le \|\Omega\chi_{\sn\backslash\mathcal D_s}\|_{L^\infty(\sn)} \le 2^{\gamma s/3}\|\Omega\|_{L^1(\sn)}.
	\end{aligned}
\end{equation}
Thus
\begin{align*}
	\widetilde K_0^s\ast K_0^s(x) &= \int_\rn\phi(y-x)\Omega(y-x)\chi_{\rn\backslash E_s}(y-x)\phi(y)\Omega(y)\chi_{\rn\backslash E_s}(y)dy \\
	&= \int_\rn\phi(x-y)\phi(y)\big(\Omega_s'(x-y)-\Omega_s''(x-y)\big)\big(\Omega_s'(y)+\Omega_s''(y)\big)dy \\
	&= (\phi\Omega_s')\ast(\phi\Omega_s')(x) - (\phi\Omega_s'')\ast(\phi\Omega_s'')(x).
\end{align*}
Reform $(\phi\Omega_s')\ast(\phi\Omega_s')$ as
\begin{align*}
	(\phi\Omega_s')\ast(\phi\Omega_s')(x) &= \int_\rn\phi(x-y)\phi(y)\Omega_s'(x-y)\Omega_s'(y)dy \\
	&= \int_\sn\int_{\mathbb R}\phi(x-t\theta)\phi(t\theta)\Omega_s'(x-t\theta)t^{n-1}dt\Omega_s'(\theta)d\sigma(\theta) \\
	&= \int_\sn L^x(\Omega_s')(\theta)\Omega_s'(\theta)d\sigma(\theta),
\end{align*}
where the second equality follows from $\Omega_s'$ is homogeneous of degree $0$, and $L^x$ was introduced in Lemma \ref{lem21}. Similarly, we have
\[(\phi\Omega_s'')\ast(\phi\Omega_s'')(x) = \int_\sn L^x(\Omega_s'')(\theta)\Omega_s''(\theta)d\sigma(\theta).\]
Therefore
\begin{align*}
	(\phi\Omega_s')\ast(\phi\Omega_s')\ast\widetilde B_{-s}(0) &= \int_\rn(\phi\Omega_s')\ast(\phi\Omega_s')(x)\widetilde B_{-s}(-x)dx \\
	&= \int_\rn\int_\sn L^x(\Omega_s')(\theta)\Omega_s'(\theta)d\sigma(\theta)\widetilde B_{-s}(-x)dx \\
	&= \int_\sn\int_\rn L^x(\Omega_s')(\theta)\widetilde B_{-s}(-x)dx\Omega_s'(\theta)d\sigma(\theta) \\
	&= \int_\sn T(\Omega_s')(\theta)\Omega_s'(\theta)d\sigma(\theta)
\end{align*}
and
\[(\phi\Omega_s'')\ast(\phi\Omega_s'')\ast\widetilde B_{-s}(0) = \int_\sn T(\Omega_s'')(\theta)\Omega_s''(\theta)d\sigma(\theta),\]
here $T$ was also introduced in Lemma \ref{lem21}. It follows from (\ref{eq3.10}) that
\begin{align*}
	|\widetilde K_0^s\ast K_0^s\ast\widetilde B_{-s}(0)| &\le \Big|\int_\sn T(\Omega_s')(\theta)\Omega_s'(\theta)d\sigma(\theta)\Big| + \Big|\int_\sn T(\Omega_s'')(\theta)\Omega_s''(\theta)d\sigma(\theta)\Big| \\
	&\le \|\Omega_s'\|_{L^\infty(\sn)}\|T(\Omega_s')\|_{L^1(\sn)} + \|\Omega_s''\|_{L^\infty(\sn)}\|T(\Omega_s'')\|_{L^1(\sn)} \\
	&\le 2^{\gamma s/3}\|\Omega\|_{L^1(\sn)}\big(\|T(\Omega')\|_{L^1(\sn)} + \|T(\Omega'')\|_{L^1(\sn)}\big).
\end{align*}
So we only need to show that
\[\|T(\Omega')\|_{L^1(\sn)} \le C2^{-2\gamma s/3}\lambda\co \quad\text{and}\quad \|T(\Omega'')\|_{L^1(\sn)} \le C2^{-2\gamma s/3}\lambda\co,\]
which is an immediate consequence of Lemma \ref{lem21}:
\[\|T(\Omega')\|_{L^1(\sn)} \le C2^{-\gamma s}\lambda\|\Omega'\|_{L^\infty(\sn)} \le C2^{-2\gamma s/3}\lambda\|\Omega\|_{L^1(\sn)} \le C2^{-2\gamma s/3}\lambda\co;\]
\[\|T(\Omega'')\|_{L^1(\sn)} \le C2^{-\gamma s}\lambda\|\Omega''\|_{L^\infty(\sn)} \le C2^{-2\gamma s/3}\lambda\|\Omega\|_{L^1(\sn)} \le C2^{-2\gamma s/3}\lambda\co.\]
The proof of Theorem \ref{thm1} is finished.
\end{proof}

\section{Proof of Theorem \ref{thm2}\label{s3}}

The next lemma plays a fundamental role in proving the limiting weak-type behaviors. It is a special version of Lemma 2.1 in \cite{QWX2021}, so we omit the proof here.

\begin{lem}[\cite{QWX2021}]\label{lem22}
	Let $\lambda>0$, $\Phi\in L^1(\sn)$, $S\subset\sn$ be a measurable set. Then we have
	\[\bigg|\bigg\{x\in\rn:\frac{|\Phi(x\big/|x|)|}{|x|^n}>\lambda,\frac x{|x|}\in S\bigg\}\bigg| = \frac{\|\Phi\|_{L^1(S)}}{n\lambda}.\]
\end{lem}

Now we state some basic properties of $L\log L$ space.

\begin{lem}[\cite{QWX2021}]\label{lem23}
	If $\Phi_1(\theta),\Phi_2(\theta)\in L\log L(\sn)$, then the following properties hold:
	\begin{enumerate}
		\item[\rm{(i)}] $\Phi_1(\theta),\Phi_2(\theta)\in L^1(\sn)$, and for $i=1,2$, $\|\Phi_i\|_{L^1(\sn)}\le\|\Phi_i\|_{L\log L(\sn)};$
		\item[\rm{(ii)}] The quasi-triangle inequality is true in $L\log L$ space:
		\[\|\Phi_1+\Phi_2\|_{L\log L(\sn)}\le4(\|\Phi_1\|_{L\log L(\sn)}+\|\Phi_2\|_{L\log L(\sn)}).\]
	\end{enumerate}
\end{lem}

\begin{proof} [Proof of Theorem \ref{thm2} (i)]
Without loss of generality, we may still assume $\Omega$ and $f$ are nonnegative, and $\|f\|_{L^1(\rn)}=1$.


For any $0<\varepsilon\ll\min\{1,\|\Omega\|_{L^1(\sn)}\}$, it is easy to see that there exists $r_{\varepsilon}>1$, such that
\[\int_{B(0,r_{\varepsilon})}f(x)dx>1-\varepsilon.\]
Now for $\lambda,\beta>0$, we denote
\begin{align*}
	F &= \{x\in\rn:M_\Omega(f)(x)>\lambda\}; \\
	F_1(\beta) &= \{x\in\rn:M_\Omega(f\chi_{B(0,r_\varepsilon)^c})(x)>\beta\lambda\}; \\
	F_2(\beta) &= \{x\in\rn:M_\Omega(f\chi_{B(0,r_\varepsilon)})(x)>\beta\lambda\}.
\end{align*}
By the sublinearity of $M_\Omega$ and the fact $M_\Omega(f\chi_{B(0,r_\varepsilon)}) \le M_\Omega(f)$, it follows that
\[F_2(1) \subset F \subset F_1(\sqrt\varepsilon) \cup F_2(1-\sqrt\varepsilon),\]
which indicates that
\begin{equation}\label{eq4.1}
	|F_2(1)| \le |F| \le |F_1(\sqrt\varepsilon)| + |F_2(1-\sqrt\varepsilon)|.
\end{equation}

For $|F_1(\sqrt\varepsilon)|$, Theorem \ref{thm1} yields that
\[\sqrt\varepsilon\lambda|F_1(\sqrt\varepsilon)| \le C\co\|\chi_{B(0,r_\varepsilon)^c}\|_{L^1(\rn)} \le C\co\varepsilon \le C\co\varepsilon^{3/4},\]
which gives that
\begin{equation}\label{eq4.2}
	|F_1(\sqrt\varepsilon)| \le C\co\frac{\varepsilon^{1/4}}\lambda.
\end{equation}
Thus by (\ref{eq4.1}) and (\ref{eq4.2}), we have
\begin{equation}\label{eq4.3}
	|F_2(1)| \le |F| \le |F_2(1-\sqrt\varepsilon)| + C\co\frac{\varepsilon^{1/4}}\lambda.
\end{equation}


To estimate $|F_2(1)|$ and $|F_2(1-\sqrt\varepsilon)|$, we need to decompose the rough kernel $\Omega$ first. Since $C(\sn)$ is dense in $L\log L(\sn)$, then there exists a nonnegative, continuous function $\Omega_\varepsilon$ on $\sn$ such that
\[\|\Omega-\Omega_\varepsilon\|_{L\log L(\sn)} < \varepsilon.\]
We extend $\Omega_\varepsilon$ from $\sn$ to $\rn$ such that it is also homogeneous of degree $0$. Furthermore, since for all $0<t<1$, $t|\log t| \le 4t^{3/4}/\text{e}$ holds, therefore
\begin{align*}
	\mathcal C_{\Omega-\Omega_\varepsilon} &\le 3\|\Omega-\Omega_\varepsilon\|_{L\log L(\sn)} + \|\Omega-\Omega_\varepsilon\|_{L^1(\sn)}\log^+\frac1{\|\Omega-\Omega_\varepsilon\|_{L^1(\sn)}} \\
	&\le 3\varepsilon + \frac4{\text{e}}\|\Omega-\Omega_\varepsilon\|_{L^1(\sn)}^{3/4} \le \left(3+\frac4{\text{e}}\right)\varepsilon^{3/4}.
\end{align*}

Now let
\begin{align*}
	F_2^1(\beta) &= \{x\in\rn:M_{\Omega-\Omega_\varepsilon}(f\chi_{B(0,r_\varepsilon)})(x)|>\beta\lambda\}; \\
	F_2^2(\beta) &= \{x\in\rn:M_{\Omega_\varepsilon}(f\chi_{B(0,r_\varepsilon)})(x)|>\beta\lambda\}.
\end{align*}
Then one can easily deduce that
\[F_2^2(1+\sqrt\varepsilon) \backslash F_2^1(\sqrt\varepsilon) \subset F_2 \qquad\text{and}\qquad F_2(1-\sqrt\varepsilon) \subset F_2^1(\sqrt\varepsilon) \cup F_2^2(1-2\sqrt\varepsilon),\]
which, with (\ref{eq4.3}), implies that
\begin{equation}\label{eq4.4}
	|F_2^2(1+\sqrt\varepsilon)| - |F_2^1(\sqrt\varepsilon)| \le F \le |F_2^1(\sqrt\varepsilon)| + |F_2^2(1-2\sqrt\varepsilon)| + C\co\frac{\varepsilon^{1/4}}\lambda.
\end{equation}

As for $|F_2^1(\sqrt\varepsilon)|$, it follows from Theorem \ref{thm1} that
\[\sqrt\varepsilon\lambda|F_2^1(\sqrt\varepsilon)| \le C\mathcal C_{\Omega-\Omega_\varepsilon}\|f\chi_{B(0,r_\varepsilon)}\|_{L^1(\rn)} \le C\varepsilon^{3/4},\]
which indicates that
\begin{equation}\label{eq4.5}
	|F_2^1(\sqrt\varepsilon)| \le C\frac{\varepsilon^{1/4}}\lambda.
\end{equation}
Therefore, by (\ref{eq4.4}) and (\ref{eq4.5}) it holds that
\begin{equation}\label{eq4.6}
	|F_2^2(1+\sqrt\varepsilon)| - C\frac{\varepsilon^{1/4}}\lambda \le F \le |F_2^2(1-2\sqrt\varepsilon)| + C(\co +1)\frac{\varepsilon^{1/4}}\lambda.
\end{equation}

Here we state the main idea of the proof of Theorem \ref{thm2} (i). To estimate $|F|$, we need to give an upper estimate of $|F_2^2(1-2\sqrt\varepsilon)|$ and a lower estimate of $|F_2^2(1+\sqrt\varepsilon)|$. We will split the rest of this proof into two parts. Due to the fact that $f\chi_{B(0,r_\varepsilon)}$ has compact support and $\Omega_\varepsilon$ is continuous, we are able to control $M_{\Omega_\varepsilon}(f\chi_{B(0,r_\varepsilon)})(x)$ when $x$ is far away from origin, this further gives the upper estimate and the lower estimate we want.

\vspace{0.2cm}


\noindent\textbf{Part 1: Upper estimate of $|F_2^2(1-2\sqrt\varepsilon)|$.}

Since $\Omega_\varepsilon$ is continuous on $\sn$, then it is uniformly continuous on $\sn$. Hence there exists $d_\varepsilon<2\varepsilon$, such that, if $\sigma(\theta_1,\theta_2)<d_\varepsilon$, we have
\[|\Omega_\varepsilon(\theta_1)-\Omega_\varepsilon(\theta_2)| < \varepsilon,\quad\text{for }\theta_1,\theta_2\in\sn.\]
Now let $R_\varepsilon=r_\varepsilon/\arcsin(d_\varepsilon)$, then it's easy to see that for $|x|>R_\varepsilon,|y|\le r_\varepsilon$, it holds that
\[\sigma\Big(\frac{x}{|x|},\frac{x-y}{|x-y|}\Big) < d_\varepsilon.\]
Therefore we can get
\begin{equation}\label{eq4.7}
	|\Omega_\varepsilon(x-y)-\Omega_\varepsilon(x)| = \Big|\Omega_\varepsilon\Big(\frac{x-y}{|x-y|}\Big)-\Omega_\varepsilon\Big(\frac{x}{|x|}\Big)\Big| < \varepsilon.
\end{equation}

Notice that when $|x|>R_\varepsilon$ and $y\in\text{supp }g$,
\begin{equation}\label{eq4.8}
	(1-\varepsilon)|x| \le \Big(1-\frac{d_\varepsilon}2\Big)|x| \le |x|-r_\varepsilon \le |x-y| \le |x|+r_\varepsilon \le \Big(1+\frac{d_\varepsilon}2\Big)|x| \le (1+\varepsilon)|x|,
\end{equation}
which means that the radius in the supremum of $M_{\Omega_\varepsilon}(f\chi_{B(0,r_\varepsilon)})(x)$ must be between $(1-\varepsilon)|x|$ and $(1+\varepsilon)|x|$. Henceforth, by (\ref{eq4.7}) and (\ref{eq4.8}), we can conclude the upper control of $M_{\Omega_\varepsilon}(f\chi_{B(0,r_\varepsilon)})(x)$ when $|x|>R_\varepsilon$:
\begin{equation}\label{eq4.9}
	\begin{aligned}
		M_{\Omega_\varepsilon}(f\chi_{B(0,r_\varepsilon)})(x) \le& \frac1{(1-\varepsilon)^n|x|^n}\int_{B(x,(1+\varepsilon)|x|)}\Omega_\varepsilon(x-y)f(y)\chi_{B(0,r_\varepsilon)}(y)dy \\
		\le& \frac1{(1-\varepsilon)^n|x|^n}\int_{B(0,r_\varepsilon)}(\Omega_\varepsilon(x)+\varepsilon)f(y)dy \\
		\le& \frac{\Omega_\varepsilon(x)+\varepsilon}{(1-\varepsilon)^n|x|^n} \le \frac{\Omega_\varepsilon(x)+\varepsilon}{(1-2\sqrt\varepsilon)^n|x|^n}.
	\end{aligned}
\end{equation}
Therefore, it follows from Lemma \ref{lem22} and (\ref{eq4.9}) that
\begin{align*}
	|F_2^2(1-2\sqrt\varepsilon)| &\le \Big|\Big\{|x|>R_\varepsilon:\frac{\Omega_\varepsilon(x)+\varepsilon}{(1-2\sqrt\varepsilon)^n|x|^n}>(1-2\sqrt\varepsilon)\lambda\Big\}\Big| + |\overline{B(0,R_\varepsilon)}| \\
	&= \frac{\|\Omega_\varepsilon+\varepsilon\|_{L^1(\sn)}}{n(1-2\sqrt\varepsilon)^{n+1}\lambda} + |\overline{B(0,R_\varepsilon)}|.
\end{align*}
Combining this estimate with (\ref{eq4.6}), we obtain the upper estimate for $|F|$:
\[|F| \le \frac{\|\Omega_\varepsilon+\varepsilon\|_{L^1(\sn)}}{n(1-2\sqrt\varepsilon)^{n+1}\lambda} + |\overline{B(0,R_\varepsilon)}| + C(\co + 1)\frac{\varepsilon^{1/4}}\lambda.\]
Now multiplying $\lambda$ on both sides of the above inequality and let $\lambda\to0^+$, we get
\begin{align*}
	\varlimsup_{\lambda\to0^+}\lambda|F| &\le \frac{\|\Omega + (\Omega_\varepsilon-\Omega) + \varepsilon\|_{L^1(\sn)}}{n(1-2\sqrt\varepsilon)^{n+1}} + C(\co + 1)\varepsilon^{1/4} \\
	&\le \frac{\|\Omega\|_{L^1(\sn)}}{n(1-2\sqrt\varepsilon)^{n+1}} + \frac{1+ \sigma(\sn)}{n(1-2\sqrt\varepsilon)^{n+1}}\varepsilon + C(\co + 1)\varepsilon^{1/4}.
\end{align*}
Hence, since $\varepsilon$ is arbitrary, it follows that
\begin{equation}\label{eq4.10}
	\varlimsup_{\lambda\to0^+}\lambda|F| \le \frac{\|\Omega\|_{L^1(\sn)}}n,
\end{equation}
which is the desired upper estimate.

\vspace{0.2cm}


\noindent{\bf Part 2: Lower estimate for $|F_2^2(1+\sqrt\varepsilon)|$.}

By (\ref{eq4.7}), we have for $|x|>R_\varepsilon$, $|y|\le r_\varepsilon$, it holds that
\[\Omega_\varepsilon(x)-\varepsilon \le \Omega_\varepsilon(x-y).\]
However, if $\Omega_\varepsilon(x)<\varepsilon$, we can not give the lower control of $M_{\Omega_\varepsilon}(f\chi_{B(0,r_\varepsilon)})$ since $\Omega_\varepsilon(x)-\varepsilon$ is negative. To overcome this obstacle, we introduce the following two auxiliary sets:
\begin{align*}
	S_\varepsilon :=& \{\theta\in\sn:\Omega_\varepsilon(\theta)>\varepsilon\}, \\
	V_\varepsilon :=& \big\{x\in\rn:\frac{x}{|x|}\in S_\varepsilon\big\}.
\end{align*}
Therefore for $x\in V_\varepsilon\cap\overline{B(0,R_\varepsilon)}^c$ and $|y|\le r_\varepsilon$, by (\ref{eq4.7}) and (\ref{eq4.8}), we can obtain the lower control of $M_{\Omega_\varepsilon}(f\chi_{B(0,r_\varepsilon)})(x)$ as follows:
\begin{equation}\label{eq4.11}
	\begin{aligned}
		M_{\Omega_\varepsilon}(f\chi_{B(0,r_\varepsilon)})(x) \ge& \frac1{(1+\varepsilon)^n|x|^n}\int_{B(x,(1+\varepsilon)|x|)}\Omega_\varepsilon(x-y)f(y)\chi_{B(0,r_\varepsilon)}(y)dy \\
		\ge& \frac1{(1+\varepsilon)^n|x|^n}\int_{B(0,r_\varepsilon)}(\Omega_\varepsilon(x)-\varepsilon)f(y)dy \\
		\ge& \frac{(1-\varepsilon)(\Omega_\varepsilon(x)-\varepsilon)}{(1+\varepsilon)^n|x|^n} \ge \frac{(1-\varepsilon)(\Omega_\varepsilon(x)-\varepsilon)}{(1+\sqrt\varepsilon)^n|x|^n}.
	\end{aligned}
\end{equation}
Lemma \ref{lem22}, together with (\ref{eq4.11}), leads to
\begin{align*}
	|F_2^2(1+\sqrt\varepsilon)| \ge& \Big|\Big\{x\in V_\varepsilon:\frac{(1-\varepsilon)(|\Omega_\varepsilon(x)|-\varepsilon)}{(1+\sqrt\varepsilon)^n|x|^n} > (1+\sqrt\varepsilon)\lambda\Big\}\Big| - |\overline{B(0,R_\varepsilon)}| \\
	=& \frac{(1-\varepsilon)\|\Omega_\varepsilon-\varepsilon\|_{L^1(S_\varepsilon)}}{n(1+\sqrt\varepsilon)^{n+1}\lambda} - |\overline{B(0,R_\varepsilon)}|,
\end{align*}
which, combining with (\ref{eq4.6}), further gives that
\[|F| \ge \frac{(1-\varepsilon)(\|\Omega_\varepsilon\|_{L^1(S_\varepsilon)} - \sigma(S_\varepsilon)\varepsilon)}{n(1+\sqrt\varepsilon)^{n+1}\lambda} - |\overline{B(0,R_\varepsilon)}| - C\co\frac{\varepsilon^{1/4}}\lambda.\]
Here we used the fact that $\|\Omega_\varepsilon-\varepsilon\|_{L^1(S_\varepsilon)} = \|\Omega_\varepsilon\|_{L^1(S_\varepsilon)} - \sigma(S_\varepsilon)\varepsilon$ since $\Omega_\varepsilon>\varepsilon$ on $S_\varepsilon$. Multiplying $\lambda$ on both sides and let $\lambda\to0^+$, we obtain
\begin{align*}
	\varliminf_{\lambda\to0^+}\lambda|F| &\ge \frac{(1-\varepsilon)(\|\Omega_\varepsilon\|_{L^1(S_\varepsilon)} - \sigma(S_\varepsilon)\varepsilon)}{n(1+\sqrt\varepsilon)^{n+1}} - C\co\varepsilon^{1/4} \\
	&= \frac{(1-\varepsilon)(\|\Omega_\varepsilon\|_{L^1(\sn)} - \|\Omega_\varepsilon\|_{L^1(\sn\backslash S_\varepsilon)} - \sigma(S_\varepsilon)\varepsilon)}{n(1+\sqrt\varepsilon)^{n+1}} - C\co\varepsilon^{1/4} \\
	&\ge \frac{(1-\varepsilon)(\|\Omega\|_{L^1(\sn)}-(1+\sigma(\sn))\varepsilon)}{n(1+\sqrt\varepsilon)^{n+1}} - C\co\varepsilon^{1/4} .
\end{align*}
It follows from the arbitrariness of $\varepsilon$ that
\begin{equation}\label{eq4.12}
	\varliminf_{\lambda\to0^+}\lambda|F| \ge \frac{\|\Omega\|_{L^1(\sn)}}n.
\end{equation}

Finally, from (\ref{eq4.10}) and (\ref{eq4.12}), it follows that
\[\lim_{\lambda\to0^+}\lambda|F|=\frac{\|\Omega\|_{L^1(\sn)}}n.\]
This completes the proof of Theorem \ref{thm2} (i).
\end{proof}


\begin{proof}[{Proof of Theorem \ref{thm2} (ii)}]
	For $\lambda,\beta>0$, we set
\begin{align*}
	G =& \Big\{x\in\rn:\Big|M_\Omega(f)(x) - \frac{\Omega(x)}{|x|^n}\Big|>\lambda\Big\}; \\
	G_1(\beta) =& \Big\{x\in\rn:\Big|\frac{\Omega(x) - \Omega_\varepsilon(x)}{|x|^n}\Big|>\beta\lambda\Big\}; \\
	G_2(\beta) =& \Big\{|x|>R_\varepsilon:\Big|M_{\Omega_\varepsilon}(f\chi_{B(0,r_\varepsilon)})(x) - \frac{\Omega_\varepsilon(x)}{|x|^n}\Big|>\beta\lambda\Big\}.
\end{align*}
Now we claim that
\begin{equation}\label{eq5.1}
	G \subset F_1(\sqrt\varepsilon) \cup F_2^1(\sqrt\varepsilon) \cup G_1(\sqrt\varepsilon) \cup G_2(1-3\sqrt\varepsilon) \cup \overline{B(0,R_\varepsilon)}.
\end{equation}

{Indeed, to show this claim is true, it suffices }to show the complementary set of right side in (\ref{eq5.1}) is contained in $G^c$. For any $x$ in the complementary set of right side, we have $|x|>R_\varepsilon$ and the following facts:
\[\Big|\frac{\Omega(x) - \Omega_\varepsilon(x)}{|x|^n}\Big| \le \sqrt\varepsilon\lambda;\]
\[\Big|M_{\Omega_\varepsilon}(f\chi_{B(0,r_\varepsilon)})(x) - \frac{\Omega_\varepsilon(x)}{|x|^n}\Big| \le (1-3\sqrt\varepsilon)\lambda;\]
\[M_\Omega(f\chi_{B(0,r_\varepsilon)^c})(x), M_{\Omega-\Omega_\varepsilon}(f\chi_{B(0,r_\varepsilon)})(x)| \le \sqrt\varepsilon\lambda.\]
From these inequalities, it's easy to deduce that
\begin{align*}
	M_\Omega(f)(x) &\le M_\Omega(f\chi_{B(0,r_\varepsilon)})(x) + \sqrt\varepsilon\lambda \le M_{\Omega_\varepsilon}(f\chi_{B(0,r_\varepsilon)})(x) + 2\sqrt\varepsilon\lambda \\
	&\le \frac{\Omega_\varepsilon(x)}{|x|^n} + (1-\sqrt\varepsilon)\lambda \le \frac{\Omega(x)}{|x|^n} + \lambda
\end{align*}
and
\begin{align*}
	M_\Omega(f)(x) &\ge M_\Omega(f\chi_{B(0,r_\varepsilon)})(x) \ge M_{\Omega_\varepsilon}(f\chi_{B(0,r_\varepsilon)})(x) - \sqrt\varepsilon\lambda \\
	&\ge \frac{\Omega_\varepsilon(x)}{|x|^n} - (1-2\sqrt\varepsilon)\lambda \ge \frac{\Omega(x)}{|x|^n} - (1-\sqrt\varepsilon)\lambda \ge \frac{\Omega(x)}{\nu_n|x|^n} - \lambda,
\end{align*}
which implies that $x\in G^c$. Therefore claim (\ref{eq5.1}) is true.

By (\ref{eq4.2}) and (\ref{eq4.5}), one immediately gets that
\begin{equation}\label{eq5.2}
	|G| \le |G_1(\sqrt\varepsilon)| + |G_2(1-3\sqrt\varepsilon)| + C(\co + 1)\frac{\varepsilon^{1/4}}\lambda + |\overline{B(0,R_\varepsilon)}|.
\end{equation}
In regard to $|G_1(\sqrt\varepsilon)|$, applying Lemma \ref{lem22}, we obtain
\[|G_1(\sqrt\varepsilon)| = \frac{\|\Omega - \Omega_\varepsilon\|_{L^1(\sn)}}{n\sqrt\varepsilon\lambda} \le \frac{\varepsilon}{n\sqrt\varepsilon\lambda} \le C\frac{\varepsilon^{1/4}}\lambda.\]
With (\ref{eq5.2}), we have
\begin{equation}\label{eq5.3}
	|G| \le |G_2(1-3\sqrt\varepsilon) \cap V_\varepsilon| + |G_2(1-3\sqrt\varepsilon) \cap V_\varepsilon^c| + C(\co + 1)\frac{\varepsilon^{1/4}}\lambda + |\overline{B(0,R_\varepsilon)}|.
\end{equation}

Now for any $x\in G_2(1-3\sqrt\varepsilon) \cap V_\varepsilon$, from (\ref{eq4.9}) and (\ref{eq4.11}), we know that $M_{\Omega_\varepsilon}(f\chi_{B(0,r_\varepsilon)})(x)$ and $\Omega_\varepsilon(x)\big/|x|^n$ are between $(1-\varepsilon)(\Omega_\varepsilon(x)-\varepsilon)\big/((1+\varepsilon)^n|x|^n)$ and $(\Omega_\varepsilon(x)+\varepsilon)\big/((1-\varepsilon)^n|x|^n)$, which implies that
\[\Big|M_{\Omega_\varepsilon}(f\chi_{B(0,r_\varepsilon)})(x) - \frac{\Omega_\varepsilon(x)|}{|x|^n}\Big| \le \frac{\Omega_\varepsilon(x)+\varepsilon}{(1-\varepsilon)^n|x|^n} - \frac{(1-\varepsilon)(\Omega_\varepsilon(x)-\varepsilon)}{(1+\varepsilon)^n|x|^n} =: \frac{I_{\Omega,\varepsilon}(x)}{|x|^n}.\]
Thus by Lemma \ref{lem22}, one may obtain
\[|G_2(1-3\sqrt\varepsilon) \cap V_\varepsilon| \le \Big|\Big\{x\in V_\varepsilon:\frac{I_{\Omega,\varepsilon}(x)}{|x|^n}>(1-3\sqrt\varepsilon)\lambda\Big\}\Big| = \frac{\|I_{\Omega,\varepsilon}\|_{L^1(S_\varepsilon)}}{n(1-3\sqrt\varepsilon)\lambda}.\]
It is easy to see that
\begin{align*}
	\|I_{\Omega,\varepsilon}\|_{L^1(S_\varepsilon)}
	&= \|\Omega_\varepsilon\|_{L^1(S_\varepsilon)}\left(\frac1{(1-\varepsilon)^n}-\frac{1-\varepsilon}{(1+\varepsilon)^n}\right) + \left(\frac1{(1-\varepsilon)^n}+\frac{1-\varepsilon}{(1+\varepsilon)^n}\right)\sigma(S_\varepsilon)\varepsilon \\
	&\le C(\|\Omega\|_{L^1(\sn)} + 1)\varepsilon \le C(\co + 1)\varepsilon^{1/4}.
\end{align*}
Therefore
\begin{equation}\label{eq5.4}
	|G_2(1-3\sqrt\varepsilon) \cap V_\varepsilon| \le C(\co + 1)\frac{\varepsilon^{1/4}}\lambda.
\end{equation}

On the other hand, if $x\in G_2(1-3\sqrt\varepsilon) \cap V_\varepsilon^c$, it follows from (\ref{eq4.9}) that
\[\Big|M_{\Omega_\varepsilon}(f\chi_{B(0,r_\varepsilon)})(x) - \frac{\Omega_\varepsilon(x)|}{|x|^n}\Big| \le 2\frac{\Omega_\varepsilon(x)+\varepsilon}{(1-\varepsilon)^n|x|^n} \le 2\frac{\Omega_\varepsilon(x)+\varepsilon}{(1-3\sqrt\varepsilon)^n|x|^n}.\]
Then by Lemma \ref{lem22}, $|G_2(1-3\sqrt\varepsilon) \cap V_\varepsilon^c|$ is dominated by
\begin{equation}\label{eq5.5}
	|G_2(1-3\sqrt\varepsilon) \cap V_\varepsilon^c| \le \frac{2\|\Omega_\varepsilon+\varepsilon\|_{L^1(\sn\backslash S_\varepsilon)}}{n(1-3\sqrt\varepsilon)^{n+1}\lambda} \le \frac{4\sigma(\sn\backslash S_\varepsilon)\varepsilon}{n(1-2\sqrt\varepsilon)^{n+1}\lambda} \le C\frac{\varepsilon^{1/4}}\lambda.
\end{equation}

Finally, by (\ref{eq5.3})-(\ref{eq5.5}), it holds that
\begin{align*}
	|G| \le C(\co + 1)\frac{\varepsilon^{1/4}}\lambda.
\end{align*}
Multiplying $\lambda$ on both sides, and by the arbitrariness of $\varepsilon$, we finally get
\[\lim_{\lambda\to0^+}\lambda|G|=0,\]
{which finishes the proof of Theorem \ref{thm2} (ii).}
\end{proof}

\bigskip

\noindent Moyan Qin

\smallskip

\noindent {\it Address:} Laboratory of Mathematics and Complex Systems (Ministry of Education of China), School of Mathematical Sciences, Beijing Normal University, Beijing 100875, People's Republic of China

\smallskip

\noindent {\it E-mail:} \texttt{myqin@mail.bnu.edu.cn}

\medskip

\noindent Huoxiong Wu

\smallskip

\noindent {\it Address:} School of Mathematical Sciences, Xiamen University, Xiamen 361005, China

\smallskip

\noindent {\it E-mail:} \texttt{huoxwu@xmu.edu.cn}

\medskip

\noindent Qingying Xue

\smallskip

\noindent {\it Address:} Laboratory of Mathematics and Complex Systems (Ministry of Education of China), School of Mathematical Sciences, Beijing Normal University, Beijing 100875, People's Republic of China

\smallskip

\noindent {\it E-mail:} \texttt{qyxue@bnu.edu.cn}

\end{document}